\documentclass[12pt,twoside,a4paper,final]{amsart}

\usepackage[margin=2.85cm,marginparwidth=2cm]{geometry}

\usepackage[utf8]{inputenc}
\usepackage[T1]{fontenc}
\usepackage{microtype}

\usepackage[]{fixme}
\fxsetface{margin}{\tiny}
\usepackage{lipsum}
\usepackage{url}
\usepackage{enumerate}
\usepackage[final]{graphicx}    

\usepackage{amsmath, amssymb, mathtools,amsfonts}
\usepackage{amsthm}
\newtheorem{theorem}{Theorem}
\newtheorem*{theorem*}{Theorem}

\newtheorem{lemma}[theorem]{Lemma}
\newtheorem{corollary}[theorem]{Corollary}
\theoremstyle{definition}
\newtheorem{definition}[theorem]{Definition}

\newcommand{\setN}{\mathbb{N}}
\newcommand{\setZ}{\mathbb{Z}}
\newcommand{\setQ}{\mathbb{Q}}
\newcommand{\setR}{\mathbb{R}}

\DeclarePairedDelimiterX{\diam}[1]{\lvert}{\rvert}{#1}
\DeclarePairedDelimiterX{\abs}[1]{\lvert}{\rvert}{#1}
\DeclarePairedDelimiterX{\cls}[1]{\lvert}{\rvert}{#1} 
\DeclarePairedDelimiterX{\norm}[1]{\lVert}{\rVert_\infty}{#1}
\DeclarePairedDelimiterX{\distz}[1]{\lVert}{\rVert}{#1}
\DeclarePairedDelimiterX{\normp}[1]{\lvert}{\rvert_p}{#1}
\DeclarePairedDelimiterX{\twonorm}[1]{\lVert}{\rVert_2}{#1}
\DeclarePairedDelimiterX\set[2]%
  {\lbrace}{\rbrace}{#1 : #2}
\DeclarePairedDelimiterX\seq[1]{\lbrace}{\rbrace}{#1}

\newcommand{\simapp}{\mathcal{S}} 
\newcommand{\myapp}{V}              
\newcommand{\intsimapp}{\mathcal{I}}
\newcommand{\Hm}{\mathcal{H}}       

\newcommand{\proj}{\mathbb{P}}
\newcommand{\Hp}{H_\textup{proj}} 
\newcommand{\leb}{\lambda}      
\newcommand{\ac}[1]{\overline{#1}}

\renewcommand{\vec}[1]{\mathbf{#1}}

\newcommand{\mysubstack}
{
  \substack{%
    \vec{p} \in \setZ^n
    \\
    0 \leq p_1,\dots,p_n \leq q
    \\
    \gcd(p_1, \dots, p_n, q) = 1
  }
}
\newcommand{\mylimsup}
{
  \bigcap_{N=1}^\infty
  \bigcup_{q > N}
  \quad
  \smashoperator{\bigcup_{\mysubstack}}
  \quad
}

\title[Intrinsic Diophantine Approximation]
{Intrinsic Diophantine Approximation\\on General Polynomial Surfaces}

\author{Morten Hein Tiljeset}

\thanks{The research was partially supported by the Danish Research Council for Independent
 Research.}

\address{M. H. Tiljeset, Department of Mathematical Sciences, Faculty
  of Science, University of Aarhus, Ny Munkegade 118,
  DK-8000 Aarhus C, Denmark}

\email{mortil@math.au.dk}

\date{\today}


\begin{document}

\begin{abstract}
  We study the Hausdorff measure and dimension of the set of
  intrinsically simultaneously $\psi$-approximable points on a curve,
  surface, etc., given as a graph of integer valued polynomials. We
  obtain complete answers to these questions for algebraically
  ``nice'' manifolds. This generalizes earlier work done in the case
  of curves.
\end{abstract}

\maketitle

\section{Introduction}
\label{sec:introduction}

Let $\vec{x} \in \setR^n$ be a real vector. Let $\setR^+$ denote the
positive real numbers. Here and throughout, let
$\psi: \setN \to \setR^+$ be a decreasing function, which we will
refer to as an \emph{approximation function}. We say that $\vec{x}$ is
\emph{(simultaneously) $\psi$-approximable} if there exist infinitely
many rational points $\vec{p}/q$ with $\vec{p} \in \setZ^n$ and
$q \in \setN$ such that
\begin{equation*}
  \norm{\vec{x} - \vec{p}/q} \leq \psi(q).
\end{equation*}
We denote the set of $\psi$-approximable vectors by $\simapp_\psi$,
and for the particular approximation functions
$\psi_\tau(r) = r^{-\tau}$ we use the notation
$\simapp_\tau = \simapp_{\psi_\tau}$.

In this notation, the classical theorem of Dirichlet states that
$\simapp_{1+1/n} = \setR^n$. We say that $\vec{x} \in \setR^n$ is
\emph{very well approximable (VWA)} if $\vec{x} \in \simapp_\tau$ for
some $\tau > 1 + 1/n$. Otherwise, we say that $\vec{x}$ is not very
well approximable or \emph{extremal}.

The theory of metric Diophantine approximation seeks to quantify, in terms of
measure, the size of $\simapp_\psi$. The starting point is the
following theorem due to Khintchine.
\begin{theorem}[Khintchine]
  \label{theorem:khintchine}
  Let $\psi: \setN \to \setR^+$ be a decreasing approximation
  function. Let $\leb_n$ denote the $n$-dimensional Lebesgue measure. Then,
  \begin{equation*}
    \leb_n(\simapp_\psi) =
    \begin{cases}
      0 &\text{ if } \sum_{r=1}^\infty r^n \psi(r)^n < \infty
      \vspace*{0.2em}
      \\
      \infty &\text { if } \sum_{r=1}^\infty r^n \psi(r)^n = \infty.
    \end{cases}
\end{equation*}
\end{theorem}
In particular, this theorem shows that almost all points are
extremal. Furthermore, it gives metric answers not only for functions
of the form $\psi_\tau(r) = r^{-\tau}$ but for general approximation
functions. In what follows, it will be useful to understand where this
theorem comes from. The starting point is to realize that we can write
the set of $\psi$-approximable vectors as a limsup set:
\begin{equation*}
    \simapp_\psi = \bigcap_{N=1}^\infty \bigcup_{q > N} \bigcup_{\vec{p} \in \setZ^n}
  B\left(\frac{\vec{p}}{q}, \psi(q)\right),
\end{equation*}
where
$B(\vec{x},r) = \set
{\vec{y} \in \setR^n}
{\norm{\vec{x} - \vec{y}} \leq r}$
is the ball in the sup-norm.

The convergence part of the theorem now follows by restricting to a
countable cover given by sets of the form $\simapp_\psi \cap I^n$
where $I \subset \setR$ is a bounded interval, and applying the first
Borel-Cantelli lemma. The hard part of the theorem is then to prove
the divergence case by establishing that the sets do not overlap too
much, so that the conclusion of the second Borel-Cantelli lemma may be
recovered.

A more refined viewpoint is given by replacing the Lebesgue measure
with the Hausdorff measure and studying in more detail the size of the
null-sets. Before proceeding, we recall the definition of the
Hausdorff measure and dimension. A complete account is available in
\cite{MR3236784}.

\begin{definition}
  \label{definition:hausdorff}
  Let $E \subseteq \setR^n$ be some set and let
  $f: \setR^+ \to \setR^+$ be an increasing and continuous function
  such that $f(r) \to 0$ as $r \to 0$, which we refer to as a
  dimension function. For any $\delta > 0$ we define
  \begin{equation*}
    \Hm^f_\delta(E) =
    \inf
    \set*{\sum f(\diam{U_i})}
         {\text{$\seq{U_i}$ is a cover of $E$ with $\diam{U_i} < \delta$}}
  \end{equation*}
  where $\diam{U_i} = \sup\set*{\abs{x-y}}{x,y \in U_i}$ is the
  diameter of $U_i$. We now define the \emph{(outer) Hausdorff
    $f$-measure} on $E$ by
  \begin{equation*}
    \Hm^f(E) = \lim_{\delta \to 0} H^f_\delta(E).
  \end{equation*}

  As a special case, for any $s \geq 0$, the \emph{Hausdorff
    $s$-measure} is the Hausdorff $f$-measure given by the dimension
  function $f(r) = r^s$ and we denote it by $\Hm^s$. It turns out
  that for any subset $E \subseteq \setR^n$ there is some number $s$
  such that $\Hm^t(E) = \infty$ for any $0 \leq t < s$ (which is an
  empty set when $s = 0$) and $\Hm^t(E) = 0$ for any $t > s$. We call
  this number the \emph{Hausdorff dimension} of $E$.
\end{definition}

There is an analogue of Khintchine's theorem for Hausdorff measures,
known as Jarn\'ik's theorem. A modern version of the theorem is the
following (see \cite[Theorem~DV]{Beresnevich:2006gk}).
\begin{theorem}[Jarn\'ik; Dickinson, Velani]
  \label{theorem:jarnik}
    Let $f$ be a dimension function such that $r^{-n}f(r) \to \infty$ as
  $r \to 0$ and $r \mapsto r^{-n}f(r)$ is decreasing. Then
  \begin{equation*}
    \Hm^f(\simapp_\psi)
    =
    \begin{cases}
      0 & \text{if } \sum_{r=1}^\infty f(\psi(r)) r^n < \infty
      \\
      \infty & \text{if } \sum_{r=1}^\infty f(\psi(r)) r^n = \infty.
    \end{cases}
  \end{equation*}
\end{theorem}
When $f(r) = r^n$ the conclusion of the theorem is the same as that of
Khintchine's theorem, but Jarn\'ik's theorem does not imply
Khintchine's theorem as $f$ does not satisfy the growth
condition. However, it is a rather surprising fact that Khintchine's
theorem implies Jarn\'ik's theorem by the Mass Transference Principle
of Beresnevich and Velani \cite{MR2259250}. The essence of this
principle is that when we're rescaling the measure, we also have to
rescale the balls in the limsup set.

Let $M \subset \setR^n$ be a manifold. Our problem is to study the
metric nature of the set
\begin{equation*}
  \simapp_\psi(M) = \simapp_\psi \cap M
\end{equation*}
This problem originates in a problem of Mahler, who conjectured that
almost all points (with respect to the induced Lebesgue measure) on
the Veronese curve
\begin{equation*}
  \mathcal{V} = \seq{(x, x^2, \dots, x^n)}
\end{equation*}
are extremal. A manifold with this property is called extremal.  The
question of Mahler was answered in the affirmative by
Sp\-r\-ind\-žuk, and the result has later been generalized to a large
class of non-degenerate manifolds by Kleinbock and Margulis using
dynamical methods and the Dani-Margulis correspondence
\cite{MR1652916}.

Having established the correct exponent for approximation on $M$, two
natural problems emerge:
\begin{enumerate}[(i)]
\item To replace functions of the form $\psi_\tau(r) = r^{-\tau}$ by
  more general approximation functions, by obtaining a Khintchine-type
  theorem for manifolds.
\item To further study the size of the null sets, by obtaining
  Hausdorff measure and dimension of the set $\simapp_\psi(M)$.
\end{enumerate}
Some general theory has been established for the first problem. For a
survey, see the monograph of Bernik and Dodson \cite{MR1727177} as
well as the more recent paper of Beresnevich \cite{MR2874641}.

It is very tempting to think that the second problem would follow from
the first by the Mass Transference Principle as in the classical
case. However, this is not quite so as we are approximating by points
\emph{outside} the manifold and we are now considering the
intersection of a limsup set with the manifold. In fact, it turns out
that the second problem is quite different from the first and depends
on the subtle arithmetic nature of the manifold. The reason for this
is the following: For many manifolds, when we have sufficiently good
approximation, the approximating points must eventually lie on the
manifold itself. To the author's knowledge, this was first observed in
\cite{MR1816807} and their argument easily generalizes to any variety
of the form $x^n + y^n = r$ where $r, n \in \setN$ are fixed natural
numbers. In the case of a circle of radius $1$ we have an abundance of
rational points, however for the circle of radius 3 or indeed the
Fermat Curve, we have only finitely many rational points and
Diophantine approximation is not possible at all.

This gives rise to another type of Diophantine approximation on
manifolds, which has gained interest in the recent years: the
question of \emph{intrinsic Diophantine approximation}. In contrast,
we will refer to the previous form of approximation as \emph{ambient
  approximation}. We introduce the notation
\begin{equation*}
  \intsimapp_\psi(M) = 
    \set{\vec{x} \in M}
    {\norm{\vec{x} - \vec{p}/q} \leq \psi(q)
      \text{ for infinitely many $\vec{p}/q \in \setQ^n \cap M$}}
\end{equation*}
for the set of intrinsically $\psi$-approximable points on $M$ and for
$\psi_\tau(r) = r^{-\tau}$ we define
$\intsimapp_\tau = \intsimapp_{\psi_\tau}$.

As there is no known method for determining whether a given variety
has infinitely many rational points, a general theory for
intrinsic Diophantine approximation is far away. However, some results
have been obtained in special cases: The case of the circle is
well-understood \cite{MR1816807}, as is the case of certain polynomial
curves \cite{Budarina:2010dd}.
More recently, results for spheres \cite{1301.0989} and
more general quadratic surfaces \cite{1405.7650} as well as homogenous
varieties \cite{MR3252026} have been obtained by dynamic methods.

Throughout we will use the Vinogradov notation, that is, for
$a, b > 0$, $a \ll b$ will mean that there exists a constant $c > 0$
such that $a \leq c b$.

\section{Statement and proof of main results}
Let $P_1, \dots, P_m \in \setZ[x_1, \dots, x_n]$ be integer
polynomials in $n$ variables, and consider a variety of the form
\begin{equation*}
  \Gamma = \set{(\vec{x}, \vec{y}) \in \setR^n \times \setR^m}
  {y_1 = P_1(\vec{x}), \dots, y_m = P_m(\vec{x})}.
\end{equation*}
Put $d_j = \deg P_j$ and let $d = \max_j d_j$ be the maximum
degree. In this paper, we aim to establish a Jarn\'ik-type zero-infinity
law for the Hausdorff measure of $\intsimapp_\psi(\Gamma)$. In the
case where $n=1$ this has been studied previously by N. Budarina,
D. Dickinson and J. Levesley in \cite{Budarina:2010dd}.
The special case of Veronese manifolds is covered in \cite[§2]{1509.05439}.
Furthermore,
in the case where the defining polynomials only depend on one variable
this has been studied by J. Schleischitz \cite{1601.02810v1}. Our
main result is the following theorem.
\begin{theorem}
  \label{theorem:main}
  Let $\psi$ be an approximation function and let $f$ be a dimension
  function such that for any $\delta > 0$ we have
  $f(\psi(\delta r)) \ll f(\psi(r))$ when $r$ is sufficiently large,
  and for any $C > 0$ we have $f(Cx) \ll f(x)$ when $x$ is
  sufficiently small. Suppose that $r^{-n}f(r) \to \infty$ as
  $r \to 0$ and $r \mapsto r^{-n}f(r)$ is decreasing. Finally, write
  $P_i = P_{i,0} + \dots + P_{i,d}$ where $P_{i,k}$ are homogenous
  polynomials of degree $k$, and suppose that the only common point of
  vanishing for $\seq{P_{i,d}}_{i=1}^m$ over $\ac{\setQ}$ is $0$. The
  Hausdorff $f$ measure of $\intsimapp_\psi(\Gamma)$ satisfies
    \begin{equation*}
    \Hm^f(\intsimapp_\psi(\Gamma)) =
    \begin{cases}
      0 & \text{if } \sum_{r=1}^\infty r^n f(\psi(r^d)) < \infty
      \\
      \infty & \text{if } \sum_{r=1}^\infty r^n f(\psi(r^d)) = \infty.
    \end{cases}
    \end{equation*}
    Furthermore, if $r^d\psi(r) \to 0$ as $r \to \infty$ we have
    $\intsimapp_\psi(\Gamma) = \simapp_\psi(\Gamma)$.
\end{theorem}
As a corollary, we derive the Hausdorff dimension of
$\intsimapp_\tau(\Gamma)$.
\begin{corollary}
  \label{corollary}
  Suppose the polynomials defining $\Gamma$ satisfy the condition of
  the theorem. For $\tau > (n+1)/nd$, the Hausdorff dimension of
  of $\intsimapp_\tau(\Gamma)$ is given by
  \begin{equation*}
    \dim \intsimapp_\tau(\Gamma) = \frac{1+n}{d\tau}.
  \end{equation*}
  Furthermore, if $\tau > d$ we also have
  $\simapp_\tau(\Gamma) = \intsimapp_\tau(\Gamma)$.
\end{corollary}
The proof will be split into three key lemmas, which establish the
case of divergence, convergence and the equality of ambient and
intrinsic approximation separately. Our approach mimicks that of the
proof of the main theorem in \cite{Budarina:2010dd}, where the novelty
in our argument is the use of algebraic geometry to obtain an upper
bound in the case of convergence.

Before proceeding, we make some reductions in order to write
$\intsimapp_\psi(\Gamma)$ as a manageable limsup set. As we are aiming
for a zero-infinity law, it suffices to show that the Hausdorff
measure of the $\psi$-approximable points are either full or null for
sets of the form
\begin{equation*}
  \Gamma_I = \seq{(\vec{x}, P_1(\vec{x}), \dots, P_m(\vec{x})) \in I^n \times \setR^m}.
\end{equation*}
where $I \subset \setR$ is some arbitrary bounded interval. For
notational simplicity we take $I = [0,1]$, however the argument does
not use this in any essential way.

Define the function $F: \setR^n \to \Gamma$ by
$F(\vec{x}) = (\vec{x}, P_1(\vec{x}), \dots, P_m(\vec{x}))$. Now, by
the mean value theorem, we can find a constant $K \geq 1$ such that
for any $\vec{x_1}, \vec{x_2} \in I^n$ we have
\begin{equation*}
  \norm{\vec{x_1} - \vec{x_2}}
  \leq
  \norm{F(\vec{x_1}) - F(\vec{x_2})}
  \leq
  K   \norm{\vec{x_1} - \vec{x_2}},
\end{equation*}
so $F$ is a bi-Lipschitz function on $I$. Since $f(Kx) \ll f(x)$ when
$x$ is sufficiently small, the Hausdorff measure is changed by at most
a constant under a bi-Lipschitz mapping. It thus suffices to show that
the measure is full or null for the set
\begin{equation*}
  \myapp_\psi(\Gamma_I) = \set{\vec{x} \in I^n}{F(\vec{x}) \in \intsimapp_\psi(\Gamma)}.
\end{equation*}

\begin{definition}
  For a rational vector $\vec{x}$ in $\setR^k$, we define the
  \emph{affine height} of $\vec{x}$ to be the least natural number
  $D$ such that
  \begin{equation*}
    \vec{x} = (r_1/D, \dots, r_k/D)
    \text{ and }
    \gcd(r_1, \dots, r_k, D) = 1
  \end{equation*}
  for some $r_1, \dots, r_k \in \setZ$.
  We also define the height function $H:\setQ^n \to \setN$ by $H(\vec{x}) = D$.
\end{definition}

We are now in a position to write $\myapp_\psi(\Gamma_I)$ as a limsup
set. Recall that $\myapp_\psi(\Gamma_I)$ consists of the set of
$\vec{x} \in I^n$ such that
$\norm{F(\vec{x}) - \vec{r}} \leq \psi(H(\vec{r}))$ for infinitely
many $\vec{r} \in \Gamma_I \cap \setQ^{n+m}$. Such rationals are
necessarily of the form $\vec{r} = F(\vec{p}/q)$ for some rational
$\vec{p}/q \in \setQ^n$. We thus have
\begin{equation}
  \begin{split}
    \label{eq:cover}
    \mylimsup B\left(\frac{\vec{p}}{q},
      \frac{\psi(H(F(\vec{p}/q)))}{K}\right) \subseteq
    \myapp_\psi(\Gamma_I)
    \\
    \myapp_\psi(\Gamma_I) \subseteq \mylimsup
    B\left(\frac{\vec{p}}{q}, \psi(H(F(\vec{p}/q)))\right).
  \end{split}
\end{equation}

\begin{lemma}[Divergence case]
  \label{lemma:divergence}
  Let $\psi$ be an approximation function. Let $f$ be a dimension
  function such that for any $C > 0$ we have $f(Cx) \ll f(x)$ when $x$
  is sufficiently small. Suppose that $r^{-n}f(r) \to \infty$ as
  $r \to 0$ and $r \mapsto r^{-n}f(r)$ is decreasing. If
  $\sum_{r=1}^\infty r^n f(\psi(r^d)) = \infty$ then
  $\Hm^f(\intsimapp_\psi(\Gamma)) = \infty$.
\end{lemma}

\begin{proof}[Proof of Lemma~\ref{lemma:divergence}]
  Let $\vec{p}/q \in I^n$ be some rational vector. It is clear that
  $q^d$ is a common multiple of the denominators in $F(\vec{p}/q)$ so
  $H(F(\vec{p}/q)) \leq q^d$. As $\psi$ is decreasing, we have
  \begin{equation*}
    \frac{\psi(H(F(\vec{p}/q)))}{K} \geq \frac{\psi(q^d)}{K}.
  \end{equation*}
  This implies that
  \begin{equation*}
    \mylimsup B\left(\frac{\vec{p}}{q}, \frac{\psi(q^d)}{K}\right)
    \subseteq V_\psi(\Gamma_I).
  \end{equation*}
  The set on the left is simply the set of
  $\phi(q) \coloneqq \psi(q^d)/K$-approximable points on $I^n$. By
  Theorem~\ref{theorem:jarnik} this set is full if
  \begin{equation*}
    \sum_{r=1}^\infty f(\phi(r))r^n = \infty.
  \end{equation*}
  If $\psi(r^d)$ does not tend to $0$ as $r \to \infty$ this is
  trivial. Otherwise, we can apply the estimate
  $f(\psi(r^d)/K) \gg f(\psi(r^d))$ when $r$ is large to obtain
  \begin{equation*}
    \sum_{r=1}^\infty f(\psi(r^d)/K) \gg
    \sum_{r=1}^\infty f(\psi(r^d)) = \infty.
  \end{equation*}
\end{proof}

\begin{lemma}[Convergence case]
  \label{lemma:convergence}
  Let $\psi$ be an approximation function and let $f$ be a dimension
  function such that for any $\delta > 0$ we have
  $f(\psi(\delta r)) \ll f(\psi(r))$ when $r$ is sufficiently large,
  and for any $C > 0$ we have $f(Cx) \ll f(x)$ when $x$ is
  sufficiently small. Write $P_i = P_{i,0} + \dots + P_{i,d}$ where
  $P_{i,k}$ are homogenous polynomials of degree $k$, and suppose that
  the only common point of vanishing for $\seq{P_{i,d}}_{i=1}^m$ over
  $\ac{\setQ}$ is $0$. If $\sum_{r=1}^\infty r^n f(\psi(r^d)) <
  \infty$ then $\Hm^f(\intsimapp_\psi(\Gamma)) = 0$.
\end{lemma}
The key ingredient in the proof of the convergence case is to get a
lower bound on the height of $F(\vec{p}/q)$. Our approach to doing
this uses projective methods from algebraic geometry, which we briefly
introduce.

Recall that \emph{projective $n$-space} over a field $k$ is the set 
$
\proj^n(k) = (k^{n+1}\setminus\seq{0}) \bigm/ \sim
$
where $P \sim Q$ when $P = \lambda Q$ for some $\lambda \in k^*$. We
can go from $k^n$ to $\proj^n(k)$ by the embedding
$(x_1, \dots, x_n) \hookrightarrow (1 : x_1 : \dots : x_n)$. As we
cannot tell the difference between numerators and denominators in
projective space, we define a new height on these spaces.
\begin{definition}
  Let $P \in \proj^n(\setQ)$. By clearing denominators, we can write
  $P$ uniquely (up to a sign) as
  \begin{equation*}
    P = (x_0, \dots, x_n)
  \end{equation*}
  where $x_0, \dots, x_n \in \setZ$ and $\gcd(x_0, \dots, x_n) =
  1$. Define the \emph{projective height} of $P$ as
  \begin{equation*}
    \Hp(P) = \max\seq[\big]{\abs{x_0}, \dots, \abs{x_n}}.
  \end{equation*}
\end{definition}
We also want to define functions between these spaces. For this we
introduce the rational maps and the morphisms. We should remark that
our definition is more restrictive than the one usually found in the
literature, but it is sufficient for our purposes.
\begin{definition}
  Let $k$ be a field.
  A map $\phi: \proj^n(k) \to \proj^m(k)$ is called a \emph{rational map} of
  degree $d$ if
  \begin{equation*}
  \phi(P) = (f_1(P), \dots, f_m(P))
\end{equation*}
where $f_1, \dots, f_m$ are homogenous polynomials of the same degree
$d$. Note that $\phi$ is only defined at $P \in \proj^n(k)$ if the polynomials
$f_1, \dots, f_m$ do not all vanish at $P$. If $\phi$ is defined
everywhere, we say that $\phi$ is a \emph{morphism over $k$}.
\end{definition}

The way we control the height is the following theorem, which is a
special case of \cite[Theorem~B.2.5]{MR1745599}.
\begin{theorem}
  \label{theorem:proj-height}
  Let $\phi: \proj^n\big(\ac{\setQ}\big) \to \proj^m\big(\ac{\setQ}\big)$
  be a morphism of degree $d$ over the
  algebraic closure of $\setQ$. For all $P \in \proj^n(\setQ)$ we have
  \begin{equation*}
    \Hp(P)^d \ll \Hp(\phi(P)) \ll \Hp(P)^d
  \end{equation*}
  where the implied constants depend on $\phi$ but not on $P$.
\end{theorem}

\begin{proof}[Proof of Lemma~\ref{lemma:convergence}]
  In order to apply Theorem~\ref{theorem:proj-height} we need to
  extend $F$ to $\proj^n$. For each of the defining polynomials
  $P_i \in \setZ[X_1, \dots, X_n]$ write
  $P_i = P_{i,0} + \dots + P_{i,d}$ where $P_{i,k}$ is a homogenous
  polynomial of degree $k$. Define the degree $d$ homogenization of
  $P_i$ by
  \begin{equation*}
    P_i^* = X_0^d P_{i,0} + X_0^{d-1} P_{i,1} + \dots + P_{i,d}.
  \end{equation*}
  Note that $P_i^* \in \setZ[X_0, \dots, X_n]$ is a homogenous
  polynomial of degree $d$. We now define the (rational) map
  $F^*: \proj^n \to \proj^{n+m}$ by
  \begin{equation*}
    F^*(X_0, \dots, X_n) = (X_0^d, X_0^{d-1} X_1, \dots, X_0^{d-1}X_n,
    P_1^*(X_0, \dots, X_n), \dots, P_m^*(X_0, \dots, X_n)).
  \end{equation*}
  In the affine patch $X_0 = 1$ this corresponds to the map $F$ from
  above. We further claim that this map is a morphism. For
  $X_0 \neq 0$ we clearly have $X_0^d \neq 0$ so it is
  well-defined. If $X_0 = 0$ the defining polynomials vanish if and
  only if $P_{1,d},\dots, P_{n,d}$ have a common point of vanishing
  away from $0$ over the algebraic closure of $\setQ$, but this does
  not happen by our assumption. Thus, $F^*$ is a morphism over the
  algebraic numbers.

  Now let $\vec{p}/q \in I^n$ be some rational vector. Since $I$ is a
  bounded interval, the ratio between the projective and affine
  heights is at most a multiplicative constant. We now have
  \begin{equation*}
    q^d \leq \Hp((1, p_1/q, \dots, p_n/q))^d
    \ll\Hp (F^*(1, p_1/q, \dots, p_n/q))
    \ll H(F(\vec{p}/q)).
  \end{equation*}
  Here the implied constants only depend on the variety $\Gamma$ and
  the interval $I$. Let $\delta > 0$ be the constant such that
  $H(F(\vec{p}/q)) \geq \delta q^d$. Now by the inclusion \eqref{eq:cover}
  and the estimate $f(\psi(\delta q^d)) \ll f(\psi(q^d))$ we have for any
  $N \in \setN$
  \begin{align*}
    \Hm^f(V_\psi(\Gamma_I))
    \ll \sum_{q > N} \sum_{\mysubstack} f(\psi(H(F(\vec{p}/q))))
    \ll \sum_{q > N} q^n f(\psi(q^d))
    < \infty.
  \end{align*}
  So the Hausdorff measure is bounded by the tail of a convergent
  series, and we conclude that $\Hm^f(V_\psi(\Gamma_I)) = 0$.
\end{proof}

Finally, the equivalence of intrinsic and ambient approximation is
given by the following lemma. This was already shown in full
generality in \cite[Lemma~1]{Budarina:2010dd} and we hence omit the
proof.
\begin{lemma}
  \label{lemma:intrinsic-is-ambient}
  Let $\psi: \setN \to \setR^+$ be an approximation function
  satisfying the growth condition $r^d \psi(r) \to 0$ as
  $r \to \infty$. Let $\vec{x} \in \simapp_\psi(\Gamma)$. If
  \begin{equation*}
    \norm{\vec{x} - \vec{r}} \leq \psi(H(\vec{r}))
  \end{equation*}
  for $\vec{r} \in \setQ^{n+m}$ with $H(\vec{r})$ sufficiently large,
  then $\vec{r} \in \Gamma$.
\end{lemma}

\begin{proof}[Proof of Theorem~\ref{theorem:main}]
  This is an immediate consequence of Lemma~\ref{lemma:divergence},
  Lemma~\ref{lemma:convergence} and
  Lemma~\ref{lemma:intrinsic-is-ambient}.
\end{proof}

\begin{proof}[Proof of Corollary~\ref{corollary}]
  Put $\psi(r) = r^{-\tau}$ and $f(r) = r^s$ where
  $s = (1+n)/(d\tau)$. We should verify that $f$ satisfies the
  conditions of the theorem. We have
  \begin{equation*}
    r^{-n} f(r) = r^{s-n} \to \infty \text{ as $r \to 0$}
  \end{equation*}
  precisely when
  \begin{equation*}
    s - n = \frac{1 + n}{d\tau} - n < 0
  \end{equation*}
  but this is satisfied as $\tau > (1+n)/dn$.

  By the strict inequality, the theorem is satisfied for
  dimension functions $f(r) = r^t$ where $t$ is in some small
  interval around $s$. It follows that
  $\Hm^t(\intsimapp_{\tau}(\Gamma)) = \infty$ when $t \leq s$
  and $\Hm^t(\intsimapp_{\tau}(\Gamma)) = 0$ when $t > s$.
\end{proof}

\section{Endnotes and examples}
\label{sec:endnotes-discussion}

The most restrictive condition in the theorem is the requirement that
the degree $d$ parts of the polynomials do not have a common point of
vanishing away from $0$ over $\ac{\setQ}$. For $n=1$ this is always
satisfied, as the polynomials are of the form $a_k x^k + \dots + a_0$
and $a_kx^k$ only vanishes at $0$.
For $n > 1$ this is only sometimes satisfied; examples include the Veronese surface and
$\Gamma = \seq{(x, y, x^2 + y^2, x^2 - y^2)}$. On the other hand, it
is never satisfied for hypersurfaces when $n > 1$: The zero locus of a
single nonconstant polynomial in $n$ variables over an algebraically closed field has
dimension $n-1$, and hence cannot be a point.

One could ask if the theorem could be generalized to the case where
this condition is not satisfied. The key ingredient is the estimate
$H(F(\vec{r})) \gg H(\vec{r})^d$ which we derive from
Theorem~\ref{theorem:proj-height}. But if $F$ does not extend to a
morphism over $\ac{\setQ}$, this theorem does not hold and explicit counterexamples can
be constructed. It is even possible that the conclusion of
Theorem~\ref{theorem:proj-height} always has counterexamples when the
map is not a morphism, though the author has not been able to prove or
disprove this. When the conclusion of
Theorem~\ref{theorem:proj-height} fails we get additional rational
points of low height on $\Gamma$, and the present argument fails. It
seems unlikely that the main theorem of this paper still holds in this
case.

\section*{Acknowledgements}

I would like to thank my advisor, Simon Kristensen, for suggesting the
problem and for productive discussions.

\end{document}